\documentclass[12pt]{amsart}
\usepackage{times}
\usepackage{amsfonts}
\usepackage{amsthm}
\usepackage{amsmath}
\usepackage{epic}
\usepackage{times}
\usepackage{verbatim}
\def\empt{\varepsilon}
\def\NN{\mathbb{N}}

\def\bbf{\boldsymbol{f}}
\def\bw{\boldsymbol{w}}

\def\cP{\mathcal{P}}

\def\cC{\mathcal{C}}
\def\cP{\mathcal{P}}

\theoremstyle{plain}
\newtheorem{theorem}{Theorem} 
\newtheorem{lemma}[theorem]{Lemma}

\newtheorem{proposition}[theorem]{Proposition}

\theoremstyle{definition}

\theoremstyle{remark}
\newtheorem{remark}[theorem]{Remark}
\newtheorem*{note}{Note}

\begin{document}

\title{A new characteristic property of rich words}

\author{Michelangelo Bucci}

\address{Dipartimento di Matematica e Applicazioni ``R. Caccioppoli", Universit\`a degli Studi di Napoli Federico II, Via Cintia, Monte S. Angelo, I-80126 Napoli, Italy}

\email{micbucci@unina.it}

\author{Alessandro De Luca}

\address{Dipartimento di Matematica e Applicazioni ``R. Caccioppoli", Universit\`a degli Studi di Napoli Federico II, Via Cintia, Monte S. Angelo, I-80126 Napoli, Italy}

\email{alessandro.deluca@unina.it}

\author{Amy Glen}

\address{LaCIM, Universit\'e du Qu\'ebec \`a Montr\'eal, C.P. 8888, succursale Centre-ville, Montr\'eal, Qu\'ebec, H3C 3P8, Canada}

\email{amy.glen@gmail.com}

\thanks{{\bf Corresponding author:} Amy Glen}


\author{Luca Q.~Zamboni}

\address{Universit\'e de Lyon, 
Universit\'e Lyon 1, 
CNRS UMR 5208 Institut Camille Jordan, 
B\^atiment du Doyen Jean Braconnier, 
43, blvd du 11 novembre 1918, 
F-69622 Villeurbanne Cedex, France}
\address{Reykjavik University, 
School of Computer Science, 
Kringlan 1, 103 Reykjavik, Iceland}

\email{luca.zamboni@wanadoo.fr}



\date{May 28, 2008}

\begin{abstract}
Originally introduced and studied by the third and fourth authors together with J.~Justin and S.~Widmer (2008), {\em rich words} constitute a new class of finite and infinite words characterized by containing the maximal number of distinct palindromes. Several characterizations of rich words have already been established. A particularly nice characteristic property is that all `complete returns' to palindromes are palindromes. In this note, we prove that rich words are also characterized by the property that each factor is uniquely determined by its longest palindromic prefix and its longest palindromic suffix. 
\end{abstract}

\maketitle

\section{Introduction}

In \cite{xDjJgP01epis}, X.~Droubay, J.~Justin, and G.~Pirillo proved that any finite word $w$ of length $|w|$ contains at most $|w| + 1$ distinct palindromes (including the empty word). Inspired by this result, the third and fourth authors together with J.~Justin and S.~Widmer recently initiated a unified study of finite and infinite words that are characterized by containing the maximal number of distinct palindromes (see~\cite{aGjJsWlZ08pali}). Such words are called {\em rich words} in view of their `palindromic richness'. More precisely, a finite word $w$ is {\em rich} if and only if it has exactly $|w| + 1$ distinct palindromic factors. For example, $abac$ is rich, whereas $abca$ is not. An infinite word is {\em rich} if all of its factors are rich.

Rich words have appeared in many different contexts; they include episturmian words, complementation-symmetric sequences, symbolic codings of trajectories of symmetric interval exchange transformations, and a certain class of words associated with $\beta$-expansions where $\beta$ is a simple Parry number. Another special class of rich words consists of S.~Fischler's sequences with ``abundant palindromic prefixes'', which were introduced and studied in~\cite{sF06pali} in relation to Diophantine approximation.  
Some other simple examples of rich words include:  non-recurrent infinite words like $abbbb\cdots$ and $abaabaaabaaaab\cdots$; the periodic infinite words: $(aab^kaabab)(aab^kaabab)\cdots$, with $k \geq 0$; the non-ultimately periodic recurrent infinite word $\psi(\bbf)$ where $\bbf = abaababaaba\cdots$ is the {\em Fibonacci word} and $\psi$ is the morphism: $a \mapsto aab^kaabab$, $b \mapsto bab$; and the recurrent, but not uniformly recurrent, infinite word generated by the morphism: $a \mapsto aba$, $b\mapsto bb$. See \cite{aGjJsWlZ08pali} for further examples and references.

Let $u$ be a non-empty factor of a finite or infinite word $w$. We say that $u$ is {\em unioccurrent} in $w$ if $u$ has exactly one occurrence in~$w$. Otherwise, if $u$ has more than one occurrence in $w$, then there exists a factor $r$ of $w$ having exactly two distinct occurrences of $u$, one as a prefix and one as a suffix. Such a factor $r$ is called a {\em complete return} to $u$ in $w$. For example, $aabcbaa$ is a complete return to $aa$ in the rich word: $aabcbaaba$. In~\cite{aGjJsWlZ08pali}, it was shown that rich words are characterized by the property that all complete returns to palindromes are palindromes.

The following proposition collects together all of the characteristic properties of rich words that were previously established in \cite{xDjJgP01epis} and \cite{aGjJsWlZ08pali}.

\begin{proposition} \label{P:rich-aGjJ07pali}  
For any finite or infinite word $w$, the following conditions are equivalent:
\begin{itemize}
\item[i)] $w$ is rich;
\item[ii)] every factor $u$ of $w$ contains exactly $|u|+1$ distinct palindromes;
\item[iii)] for each factor $u$ of $w$, every prefix (resp.~suffix) of $u$ has a unioccurrent palindromic suffix (resp.~prefix); 
\item[iv)] every prefix of $w$ has a unioccurrent palindromic suffix; 
\item[v)] for each palindromic factor $p$ of $w$, every complete return to $p$ in $w$ is a palindrome.
\end{itemize}
\end{proposition}
\begin{remark}  \label{R:1} The equivalences:  i) $\Leftrightarrow$ ii), i) $\Leftrightarrow$ iii), and i) $\Leftrightarrow$ iv) were proved in \cite{xDjJgP01epis}. 
\end{remark}

Explicit characterizations of periodic rich infinite words and recurrent {\em balanced} rich infinite words have also been established in \cite{aGjJsWlZ08pali}. More recently, 
we proved the following connection between palindromic richness and {complexity}. 

\begin{proposition} \label{P:mBaDaGlZ08acon-1} {\em \cite{mBaDaGlZ08acon}}
For any infinite word $\bw$ whose set of factors is closed under reversal, the following conditions are  equivalent:
\begin{itemize}
\item all complete returns to palindromes are palindromes;
\item $\cP(n) + \cP(n+1) = \cC(n+1) - \cC(n) + 2$ for all $n \in \NN$,
\end{itemize}
where $\cP$ (resp.~$\cC$) denotes the {\em palindromic complexity} (resp.~{\em factor complexity})  function of $\bw$, which counts the number of distinct palindromic factors (resp.~factors) of each length in $\bw$.
\end{proposition}

From the perspective of richness, the above proposition can be viewed as a characterization of {\em recurrent} rich infinite words since any rich infinite word is recurrent if and only if its set of factors is closed under reversal (see~\cite{aGjJsWlZ08pali}). Interestingly, the proof of Proposition~\ref{P:mBaDaGlZ08acon-1} relied upon another characterization of rich words, stated below. 

\begin{proposition} \label{P:mBaDaGlZ08acon-2} {\em \cite{mBaDaGlZ08acon}}
 A finite or infinite word $w$ is rich if and only if, for each factor $v$ of $w$, every factor of $w$ beginning with $v$ and ending with $\tilde v$ and containing no other occurrences of $v$ or $\tilde v$ is a palindrome.
\end{proposition}

In this note, we establish yet another interesting characteristic property of rich words. Our main results are the following two theorems.

\begin{theorem} \label{T:main}
For any finite or infinite word $w$, the following conditions are equivalent:
\begin{itemize}
\item[(A)] $w$ is rich;
\item[(B)] each non-palindromic factor $u$ of $w$ is uniquely determined by a pair $(p,q)$ of distinct palindromes such that $p$ and $q$ are not factors of each other and $p$ (resp.~$q$) is the longest palindromic prefix (resp.~suffix) of $u$.
\end{itemize}
\end{theorem}

\smallskip
\begin{theorem} \label{T:main-2}
A finite or infinite word $w$ is rich if and only if each factor of $w$ is uniquely determined by its longest palindromic prefix and its longest palindromic suffix. 
\end{theorem}

\section{Terminology and notation} \label{S:preliminaries}

Given a finite word $w = x_1x_2\cdots x_m$ (where each $x_i$ is a letter), the  \emph{length} of $w$, denoted by $|w|$, is equal to $m$.  We denote by $\tilde w$ the {\em reversal} of $w$, given by $\tilde w = x_m \cdots x_2x_1$ (the ``mirror image'' of $w$). If $w = \tilde w$, then $w$ is called a {\em palindrome}. By convention, the {\em empty word} $\empt$ is assumed to be a palindrome.

 A finite word $z$ is a {\em factor} of a finite or infinite word $w$ if $w = uzv$ for some words $u$, $v$. In the special case $u = \empt$ (resp.~$v = \empt$), we call $z$ a {\em prefix} (resp.~{\em suffix}) of $w$. If $u \ne \empt$ and $v \ne \empt$, then we say that $z$ is an {\em interior factor} of $w =uzv$. A {\em proper factor} (resp.~{\em proper prefix}, {\em proper suffix}) of a word $w$ is a factor (resp.~prefix, suffix) of $w$ that is shorter than $w$.

\section{Proof of Theorem~\ref{T:main}} \label{S:proof}

The following two lemmas establish that (A) implies (B).

\begin{lemma} \label{L:1} Suppose $w$ is a finite or infinite rich word and let $u$ be any non-palindromic factor of $w$ with longest palindromic prefix $p$ and longest palindromic suffix $q$. Then 
$p \ne q$, and $p$ and $q$ are not factors of each other. 
\end{lemma}
\begin{proof} By Proposition~\ref{P:rich-aGjJ07pali}, $p$ and $q$ are unioccurrent factors of $u$. Thus, since $u$ is not a palindrome (and hence $|u| > \max\{|p|,|q|\}$), it follows immediately that $p \ne q$, and $p$ and $q$ are not factors of each other. 
\end{proof}

\begin{lemma} \label{L:2}
Suppose $w$ is a finite or infinite rich word. If $u$ and $v$ are factors of $w$ with the same longest palindromic prefix $p$ and the same longest palindromic suffix $q$, then  $u = v$.
\end{lemma}
\begin{proof}
We first observe that if $u$ or $v$ is a palindrome, then $u = p = q = v$. So let us now assume that neither $u$ nor $v$ is a palindrome. 

Suppose to the contrary that $u \ne v$. Then $u$ and $v$ are clearly not factors of each other since neither $u$ nor $v$ is equal to $p$ or $q$, and $p$ and $q$ are unioccurrent in each of $u$ and $v$ (by Proposition~\ref{P:rich-aGjJ07pali}). Let $z$ be a factor of $w$ of minimal length containing both $u$ and $v$. As $u$ and $v$ are not factors of each other, we may assume without loss of generality that $z$ begins with $u$ and ends with $v$. 
Then $z$ contains {\em at least} two distinct occurrences of $p$ (as a prefix of each of $u$ and $v$). In particular, $z$ begins with a complete return $r_1$ to $p$ with $|r_1| > |u|$ because $p$ is unioccurrent in $u$ by Proposition~\ref{P:rich-aGjJ07pali}. Moreover, $r_1$ is a palindrome by the richness of $w$, 
and hence $r_1$ ends with $\tilde{u}$ since $u$ is a proper prefix of $r_1$. 
Similarly, $z$ ends with a complete return $r_2$ to $q$ with $|r_2| > |v|$ since $q$ is unioccurrent in $v$ by Proposition~\ref{P:rich-aGjJ07pali}. Hence, since $r_2$ is a palindrome (by the richness of $w$) and $v$ is a proper prefix of $r_2$, it follows that $r_2$ begins with $\tilde v$. So we have shown that $\tilde u$ and $\tilde v$ are (distinct) interior factors of $z$. 

Let us first suppose that an occurrence of $\tilde v$ is followed by an occurrence of $\tilde u$ in $z$ (i.e., $z$ has an interior factor beginning with $\tilde v$ and ending with $\tilde u$). 
Then, since $q$ is a unioccurrent prefix of each of the (distinct) factors $\tilde v$ and $\tilde u$, we deduce that $z$ contains (as an interior factor) a complete return $r_3$ to $q$ beginning with $\tilde v$. 
In particular, as $r_3$ is a palindrome (by richness), 
$r_3$ ends with $v$.  
Thus, $z$ has a proper prefix beginning with $u$ and ending with $v$, contradicting the minimality of~$z$.  
On the other hand, if $z$ has an interior factor beginning with $\tilde u$ and ending with $\tilde v$, then using the same reasoning as above, we deduce that $z$ has a proper suffix beginning with $u$ and ending with $v$. But again, this contradicts the minimality of~$z$; whence $u = v$.
\end{proof}

\newpage
The proof of ``(A) $\Rightarrow$ (B)'' is now complete. The next lemma proves that (B) implies (A).

\begin{lemma} \label{L:3} Suppose $w$ is a finite or infinite word with the property that each non-palindromic factor $u$ of $w$ is uniquely determined by a pair  $(p,q)$ of distinct palindromes such that $p$ and $q$ are not factors of each other and $p$ (resp.~$q$) is the longest palindromic prefix (resp.~suffix) of $u$. Then $w$ is rich.
\end{lemma}
\begin{proof}
To prove that $w$ is rich, it suffices to show that each prefix of $w$ has a unioccurrent palindromic suffix (see Proposition~\ref{P:rich-aGjJ07pali}). 

Let $u$ be any prefix of $w$ and let $q$ be the longest palindromic suffix of $u$. We first observe that if $u$ is a palindrome then $u = q$, and hence $q$ is unioccurrent in $u$. Now let us suppose that $u$ is not a palindrome and let $p$ be the longest palindromic prefix of $u$. If $q$ is not unioccurrent in $u$, then, as $p$ and $q$ are not factors of each other (by the given property of $w$), we deduce that $u$ has a {\em proper} factor $v$ beginning with $p$ and ending with $q$ and not containing $p$ or $q$ as an interior factor. 
Moreover, we observe that $p$ is the longest palindromic prefix of $v$; otherwise $p$ would occur in the interior of $v$ (as a suffix of a longer palindromic prefix of $v$).  Similarly, we deduce that $q$ is the longest palindromic suffix of $v$. So $v$ has the same longest palindromic prefix and the same longest palindromic suffix as $u$, a contradiction. Whence $q$ is unioccurrent in $u$. This completes the proof of the lemma.
\end{proof}

\medskip
\begin{note}
Likewise, in the case when $w$ is finite, one can easily show that each suffix of $w$ has a unioccurrent palindromic prefix; whence $w$ is rich by Proposition~\ref{P:rich-aGjJ07pali}.
\end{note}

\section{Proof of Theorem~\ref{T:main-2}}

Lemma~\ref{L:2} proves that each factor of a rich word is uniquely determined by its longest palindromic prefix and its longest palindromic suffix. 

Conversely, suppose $w$ is a finite or infinite word with the property that each factor of $w$ is uniquely determined by its longest palindromic prefix and its longest palindromic suffix. To prove that $w$ is rich, we could use very similar reasoning as in the proof of Lemma~\ref{L:3}. But for the sake of interest, we give a slightly different proof. Specifically, we show that all complete returns to any palindromic factor of $w$ are palindromes; whence $w$ is rich by Proposition~\ref{P:rich-aGjJ07pali}.

Let $p$ be any palindromic factor of $w$ and let us suppose to the contrary that $w$ contains a non-palindromic complete return $r$ to $p$. Then $r \ne pp$ and the two occurrences of $p$ in $r$ cannot overlap. Otherwise $r = pz^{-1}p$ for some word $z$ such that $p = zf = gz = \tilde z \tilde g = \tilde p$; whence $z = \tilde z$ and $r = g\tilde z \tilde g = g z \tilde g$, a palindrome. So $r = pvp$ for some non-palindromic word $v$. We easily see that $p$ is both the longest palindromic prefix and the longest palindromic suffix of $r$; otherwise $p$ would occur in the interior of $r$ as a suffix of a longer palindromic prefix of $r$, or as a prefix of a longer palindromic suffix of $r$. As $r \ne p$, we have reached a contradiction to the fact that $p$ is the {\em only} factor of $w$ having itself as both its longest palindromic prefix and its longest palindromic suffix.  Thus, all complete returns to $p$ in $w$ are palindromes. 
This completes the proof of Theorem~\ref{T:main-2}. \qed


\bigskip\bigskip
\begin{thebibliography}{99}


\bibitem{mBaDaGlZ08acon} M.~Bucci, A.~De~Luca, A.~Glen, L.Q.~Zamboni, A connection between palindromic and factor complexity using return words, {\em Adv. in Appl. Math.}, to appear, arXiv:0802.1332.

\bibitem{xDjJgP01epis} X.~Droubay, J.~Justin, G.~Pirillo, Episturmian words and some constructions of de Luca and Rauzy, {\em Theoret. Comput. Sci.} 255 (2001) 539--553. 

\bibitem{sF06pali} S.~Fischler, Palindromic prefixes and episturmian words,
{\it J. Combin.  Theory Ser.~A} 113 (2006) 1281--1304.

\bibitem{aGjJsWlZ08pali} A.~Glen, J.~Justin, S.~Widmer, L.Q.~Zamboni, Palindromic richness, {\it European J. Combin.} (in press), doi:10.1016/j.ejc.2008.04.006


\end{thebibliography}
\end{document}